\documentclass[12pt]{amsart}
\usepackage{a4wide}
\usepackage[all]{xy}
\usepackage{pdfsync}
 \usepackage[colorlinks=false]{hyperref}
\usepackage{amssymb}
\usepackage{amsmath}
\usepackage{autonum}
\usepackage{graphicx}
\usepackage[english]{babel}
\usepackage{epsfig}
\usepackage{multirow}
\usepackage{color}
\usepackage{enumerate}
\usepackage[T1]{fontenc}
\usepackage{aurical, pbsi,calligra}
\swapnumbers
\usepackage{amsmath,systeme}
\usepackage{amsmath} 

\theoremstyle{definition}
\newtheorem{definition}{Definition}[section]
\newtheorem{remark}[definition]{Remark}

\theoremstyle{plain}
\newtheorem{lemma}[definition]{Lemma}

\newtheorem{theorem}[definition]{Theorem}
\usepackage[all]{xy}
\usepackage{pdfsync}
\usepackage{ytableau}

\begin{document} 

\title[NONLINEAR STEFAN PROBLEM FOR ONE-PHASE GENERALIZED HEAT EQUATION]{NONLINEAR STEFAN PROBLEM FOR ONE-PHASE GENERALIZED HEAT EQUATION WITH HEAT FLUX AND CONVECTIVE BOUNDARY CONDITION}

\author[T.A. Nauryz]{T.A. Nauryz}

\address{Kazakh British Technical University, Almaty, Kazakhstan}

\address{Institute of Mathematics and Mathematical Modeling, Almaty, Kazakhstan}

\email{targyn.nauryz@gmail.com}



\maketitle

\begin{abstract}
In this article we consider a mathematical model of an initial stage of closure electrical contact that involves a metallic vaporization after instantaneous exploding of contact due to  arc ignition with power $P_0$ on fixed face $z=0$ and heat transfer in material with a variable cross section, when the radial component of the temperature gradient can be neglected in comparison with the axial component with heat flux and convective boundary conditions prescribed at the known free boundary $z=\alpha(t)$. The temperature field in the liquid region of such kind of material can be modelled by Stefan problem for the generalized heat equation. The method of solution is based on similarity variable, which enables us to reduce generalized heat equation to nonlinear ordinary differential equation. Moreover, we have to determine temperature solution for the liquid phase and location of melting interface. Existence and uniqueness of the solution is proved by using the fixed point Banach theorem. The solution for two cases of thermal coefficients, in particular, constant and linear thermal conductivity are represented, existence and uniqueness for each type of solution is proved.
\medskip

\end{abstract}

\section{Introduction}
The heat transfer Stefan problems such as melting and freezing, diffusion process constitute a vast area with a wide engineering and industrial applications. Stefan problems describe the heat processes in phase transitions, where these phase transitions are characterized by thermal diffusion and they have been studied widely in \cite{1}-\cite{9}. The extensive bibliography related to this study is represented in \cite{10}. 

The classical direct Stefan problems with free boundaries is the phase-change problem where temperature field in liquid (in melting problem) or solid regions (in solidification problem) and interface melting temperature at free boundary $x=\beta(t)$ have to be determined but if dynamics of heat flux has to determined in this case inverse Stefan problem is considered. Such kind of problems for materials with spherical, cylindrical and cross-section domain arising in electrical contact phenomena are successfully discussed in \cite{11}-\cite{18}. Mathematical modeling of non-classical Stefan problem should take into account temperature dependence of the thermal conductivity because it is very essential to get correct description of the boiling and melting dynamics. The nonlinear Stefan problem with Diriclet, Neumann and Robin conditions on the fixed and moving face are considered and successfully solved in \cite{19}-\cite{24}. Bollati, Briozzo and Natale successfully discussed about inverse non-classical Stefan problem in which unknown thermal coefficients have to be determined \cite{25} and Briozzo, Natale with Tarzia considered inverse non-classical Stefan problem  for Storm’s-type materials through a phase-change process \cite{26}. Huntul and Lesnic also discussed an inverse problem of determining the time-dependent
thermal conductivity and the transient temperature satisfying the heat equation with boundary data \cite{27}.

\begin{figure}[t]\label{Fig1}
\includegraphics[width=7.5 cm]{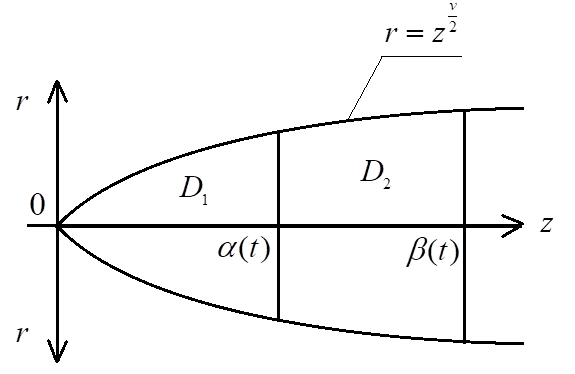}
\caption{Mathematical model of material with variable cross-section: $D_1$-metallic vapour region, $D_2$-melting region}
\end{figure}

Mathematical model of the heat transfer process in the material with cross-section variable region can be represented by the generalized heat equations. This kind of model is very useful to describe dynamics of temperature in metal bridge in electrical contact phenomena to prevent contact explosion. The mathematical model of initial stage of closure electrical contacts involves domains metallic vapour and liquid regions, see Figure \ref{Fig1}.Modeling of the temperature field in domain $D_1$ is a difficult problem, thus we suggest that heat is distributed in parabolic form and the mathematical model for a metallic vapour zone can be represented
\begin{equation}\label{1}
    \theta_1(z,t)=Az^2+Bz+C,\;\;\;0<z<\alpha(t),\;\;t>0,
\end{equation}
and temperature in this region decreases from the temperature $\theta_{im}$ which is required for ionization of the metallic vapour
\begin{equation}\label{2}
    \theta_1(0,t)=\theta_{im},\;\;\;t>0
\end{equation}
\begin{equation}\label{4}
    \theta_1(\alpha(t),t)=\theta_b,\;\;\;t>0,
\end{equation}
and the balance of heat flux on $z=\alpha(t)$ is
\begin{equation}\label{5}
    -\lambda\dfrac{\partial\theta_1}{\partial z}\bigg|_{z=\alpha(t)}=\dfrac{P_0}{2\sqrt{\pi t}}-l_b\gamma_b\dfrac{d\alpha}{dt},\;\;\;t>0,
\end{equation}
where $\theta_1(z,t)$ is a temperature in metallic vapour zone, $\theta_b$ is a boiling temperature. $P_0$ is a given positive constant, $l_b$ is a latent heat of boiling and $\gamma_b>0$ is a density of material at boiling. The location of the boiling interface $\alpha(t)$ can be represented
\begin{equation}\label{6}
    \alpha(t)=2\alpha_0\sqrt{t}.
\end{equation}

From conditions \eqref{2},\eqref{4} we can determine 
\begin{equation}\label{7}
    C=\theta_{im},\;\;\;B=0,\;\;\;A=\dfrac{1}{4\alpha_0^2t}(\theta_b-\theta_{im}).
\end{equation}

Then by using \eqref{7} the temperature field for metallic vapour zone \eqref{1} can be rewritten 
\begin{equation}\label{8}
    \theta_1(z,t)=\dfrac{z^2}{4\alpha_0^2t}(\theta_b-\theta_{im})+\theta_{im}.
\end{equation}
With the help of \eqref{8} we can easily see that the solution of the equation \eqref{5} is \eqref{6} where $\alpha_0$ can be determined from the equation 
\begin{equation}\label{9}
    \alpha_0^2+D\alpha_0+E=0
\end{equation}
where 
$$D=\dfrac{P_0}{2l_b\gamma_b\sqrt{\pi}},\;\;\;E=\dfrac{\lambda(\theta_b-\theta_{im})}{l_b\gamma_b}.$$

A mathematical model of temperature field of the domain $D_1$ can be represented
\begin{equation}\label{10}
    c(\theta_2)\rho(\theta_2)\dfrac{\partial\theta_2}{\partial t}=\dfrac{1}{z^{\nu}}\dfrac{\partial}{\partial z}\bigg[\lambda(\theta_2)z^{\nu}\dfrac{\partial\theta_2}{\partial z}\bigg],\;\;\alpha(t)<z<\beta(t),\;\;0<\nu<1,\;\;t>0,
\end{equation}
\begin{equation}\label{11}
    -\lambda(\theta_2(\alpha(t),t))\dfrac{\partial\theta_2}{\partial z}\bigg|_{z=\alpha(t)}=\dfrac{P_0e^{-\alpha_0^2}}{2\sqrt{\pi t}},\;\;\;t>0,
\end{equation}
\begin{equation}\label{12}
    \theta_2(\beta(t),t)=\theta_m,\;\;\;t>0,
\end{equation}
\begin{equation}\label{13}
    -\lambda(\theta_2(\beta(t),t))\dfrac{\partial\theta_2}{\partial z}\bigg|_{z=\beta(t)}=l_m\gamma_m\dfrac{d\beta}{dt},\;\;\; t>0,
\end{equation}
\begin{equation}\label{14}
    \beta(0)=0
\end{equation}
where $c(\theta_2)$, $\rho(\theta_2)$ and $\lambda(\theta_2)$ are  specific heat, material's density and thermal conductivity depended on temperature, $\theta_2(z,t)$ - temperature in liquid phase, $P_0$ is a given positive constant, $\theta_m$ - melting temperature, $l_m$ - latent heat of melting, $\gamma_m$ - density of material at melting, $\alpha(t)$ is a known free boundary that can be determined from \eqref{5} and \eqref{9}, $\beta(t)$ - location of the melting interface which has to be found.

We will consider one more problem replacing the heat flux condition \eqref{11} with convective boundary condition on the known free boundary $z=\alpha(t)$ such that
\begin{equation}\label{11a}
    \lambda(\theta_2(\alpha(t),t))\dfrac{\partial\theta_2}{\partial z}\bigg|_{z=\alpha(t)}=\dfrac{q}{2\sqrt{\pi t}}(\theta_2(\alpha(t),t)-\theta^*),\;\;\;t>0
\end{equation}
where $q=P_0e^{-\alpha_0^2}$ is the coefficient characterizes the heat transfer at the free boundary $z=\alpha(t)$ determined from \eqref{9}, $\theta^*$ is the reference bulk temperature which arising near to free boundary $z=\alpha(t)$ with $\theta^*>\theta_2(\alpha(t),t)$.    

The purpose of the paper is a providing similarity solution of the one-phase Stefan problem for generalized heat equation if heat flux enters to liquid region from known free boundary $z=\alpha(t)$ where boiling process starts and determination of the location of the melting interface on the boundary $z=\beta(t)$. In Section 2, similarity solution of two problems are introduced where condition \eqref{11a} replaced with \eqref{11} and this special method enables us to reduce the problem \eqref{10}-\eqref{14} boundary value problem with ordinary nonlinear differential equation. In Section 3, the existence and uniqueness of the similarity solutions of the two problems imposed \eqref{10}-\eqref{14} with two free boundaries is provided  by using fixed point Banach theorem. In the last section, we provide the solutions for particular cases of thermal coefficients and their existence, uniqueness are discussed.

\section{Similarity solution of the problem}
\subsection{Heat flux condition}

Using dimensionless transformation
\begin{equation}\label{15}
    T(z,t)=\dfrac{\theta(z,t)-\theta_m}{\theta_m}
\end{equation}
Then problem \eqref{10},\eqref{11},\eqref{12},\eqref{13} and \eqref{14} can be rewritten as
\begin{equation}\label{16}
    \Bar{N}(T_2)\dfrac{\partial T_2}{\partial t}=\dfrac{a}{z^{\nu}}\dfrac{\partial}{\partial z}\bigg[\Bar{L}(T_2)z^{\nu}\dfrac{\partial T_2}{\partial z}\bigg],\;\;\alpha(t)<z<\beta(t),\;\;0<\nu<1,\;\;t>0,
\end{equation}
\begin{equation}\label{17}
    \Bar{L}(T_2(\alpha(t),t))\dfrac{\partial T_2}{\partial z}\bigg|_{z=\alpha(t)}=-\dfrac{P_0e^{-\alpha_0^2}}{2\lambda_0\theta_m\sqrt{\pi t}},\;\;\;t>0,
\end{equation}
\begin{equation}\label{18}
    T_2(\beta(t),t)=0,\;\;\;t>0,
\end{equation}
\begin{equation}\label{19}
    \Bar{L}(T_2(\beta(t),t))\dfrac{\partial T_2}{\partial z}\bigg|_{z=\beta(t)}=-\dfrac{l_m\gamma_m}{\lambda_0\theta_m}\dfrac{d\beta}{dt},\;\;t>0,
\end{equation}
\begin{equation}\label{20}
    \beta(0)=0
\end{equation}
where 
\begin{equation}\label{NL}
    \Bar{N}(T_2)=\dfrac{c(\theta_m T_2+\theta_m)\rho(\theta_m T_2+\theta_m)}{c_0\rho_0},\;\;\;\Bar{L}(T_2)=\dfrac{\lambda(\theta_m T_2+\theta_m)}{\lambda_0}
\end{equation}
and $c_0,\;\rho_0,\;\lambda_0$, $a=\lambda_0/(c_0\rho_0)$ are heat capacity, density, thermal conductivity and thermal diffusivity of the material.

To solve problem \eqref{16},\eqref{17},\eqref{18},\eqref{19} and \eqref{20} we use similarity type substitution
\begin{equation}\label{21}
    T_2(z,t)=u_2(\eta),\;\;\;\eta=\dfrac{z}{2\sqrt{t}},
\end{equation}
and from \eqref{16},\eqref{17}, \eqref{18} and \eqref{21} the free boundaries can be represented as
\begin{equation}\label{22}
    \alpha(t)=2\alpha_0\sqrt{t},\;\;\;\beta(t)=2\xi\sqrt{t},
\end{equation}
where $\alpha_0$ is known constant which determined from \eqref{9} and $\xi$ has to be determined.
Then we obtain the following problem
\begin{equation}\label{23}
    [L^*(u_2)\eta^{\nu}u_2']'+\dfrac{2}{a}\eta^{\nu+1}N^*(u_2)u_2'=0,\;\;\alpha_0<\eta<\xi,\;\;\;0<\nu<1,
\end{equation}
\begin{equation}\label{24}
    L^*(u_2(\alpha_0))u_2'(\alpha_0)=-q^*,
\end{equation}
\begin{equation}\label{25}
    u_2(\xi)=0,
\end{equation}
\begin{equation}\label{26}
    u_2'(\xi)=-M\xi
\end{equation}
where $q^*=\dfrac{P_0e^{-\alpha_0^2}}{\alpha_0\theta_m\sqrt{\pi}}$, $M=\dfrac{2l_m\gamma_m}{\lambda_0\theta_m\lambda(\theta_m)}$ and
\begin{equation}\label{NL2}
L^*(u_2)=\dfrac{\lambda(\theta_m u_2+\theta_m)}{\lambda_0},\;\;\; N^*(u_2)=\dfrac{c(\theta_m u_2+\theta_m)\rho(\theta_m u_2+\theta_m)}{c_0\rho_0}.    
\end{equation}

We can deduce that $(u_2, \xi)$ is a solution of the problem \eqref{23},\eqref{24},\eqref{25} and \eqref{26} if and only if it satisfies the integral equation
\begin{equation}\label{27}
    u_2(\eta)=q^*[\Phi(\xi,u_2(\xi))-\Phi(\eta,u_2(\eta))]
\end{equation}
where
\begin{equation}\label{28}
    \Phi(\eta,u_2(\eta))=\alpha_0^{\nu}\int\limits_{\alpha_0}^{\eta}\dfrac{E(s,u_2(s))}{v^{\nu}L^*(u_2(v))}dv
\end{equation}
\begin{equation}\label{29}
    E(\eta,u_2(\eta))=\exp\Bigg(-\dfrac{2}{a}\int\limits_{\alpha_0}^{\eta}s\dfrac{N^*(u_2(s))}{L^*(u_2(s))}ds\Bigg)
\end{equation}
and condition 
\begin{equation}\label{30}
    \dfrac{q^*\alpha_0^{\nu}E(\xi,u_2(\xi))}{M\lambda(\theta_m)}=\xi^{\nu+1}
\end{equation}
From expression \eqref{30} we can determine $\xi$ for the free boundary $\beta(t)$. 

The solution of the free boundary \eqref{10}-\eqref{14} is given by \eqref{15} and
$$\theta_2(z,t)=\theta_m+\theta_m u_2(\eta)$$
where $\eta=z/(2\sqrt{t})$ and function $u_2(\eta)$ must satisfy the integral equation \eqref{27} and condition \eqref{30}.

\subsection{Convective boundary condition} 
If we use the dimensionless substitution
\begin{equation}\label{e1}
    T(z,t)=\dfrac{\theta(z,t)-\theta^*}{\theta_m-\theta^*}>0,
\end{equation}
then problem \eqref{10}-\eqref{14} with replaced condition with \eqref{11a} instead of heat flux condition becomes
\begin{equation}\label{e2}
    \Bar{N}(T_2)\dfrac{\partial T_2}{\partial t}=\dfrac{a}{z^{\nu}}\dfrac{\partial}{\partial z}\bigg[\Bar{L}(T_2)z^{\nu}\dfrac{\partial T_2}{\partial z}\bigg],\;\;\alpha(t)<z<\beta(t),\;\;0<\nu<1,\;\;t>0,
\end{equation}
\begin{equation}\label{e3}
    \Bar{L}(T_2(\alpha(t),t))\dfrac{\partial T_2}{\partial z}\bigg|_{z=\alpha(t)}=\dfrac{q}{2\lambda_0\sqrt{\pi t}}T_2(\alpha(t),t),\;\;\;t>0,
\end{equation}
\begin{equation}\label{e4}
    T_2(\beta(t),t)=1,\;\;\;t>0,
\end{equation}
\begin{equation}\label{e5}
    \Bar{L}(T_2(\beta(t),t))\dfrac{\partial T_2}{\partial z}\bigg|_{z=\beta(t)}=\dfrac{\beta'(t)}{a\text{Ste}},\;\;t>0,
\end{equation}
\begin{equation}\label{e6}
    \beta(0)=0
\end{equation}
where $q=P_0e^{-\alpha_0^2}$, $\text{Ste}=\frac{(\theta_m-\theta^*)c_0}{l_m}>0$ and $\Bar{N},\;\bar{L}$ are defined from \eqref{NL}.

Then using similarity transformation \eqref{22} problem \eqref{e2},\eqref{e3},\eqref{e4},\eqref{e5},\eqref{e6} can be rewritten as
\begin{equation}\label{e6}
    [L^*(u_2)\eta^{\nu}u_2']'+\dfrac{2}{a}\eta^{\nu+1}N^*(u_2)u_2'=0,\;\;\alpha_0<\eta<\xi,\;\;\;0<\nu<1,
\end{equation}
\begin{equation}\label{e7}
    L^*(u_2(\alpha_0))u_2'(\alpha_0)=p^*u_2(\alpha(t)),
\end{equation}
\begin{equation}\label{e8}
    u_2(\xi)=1,
\end{equation}
\begin{equation}\label{e9}
    L^*(u_2(\xi))u_2'(\xi)=\dfrac{2\xi}{a\text{Ste}}
\end{equation}
where $p^*=q/(\lambda_0\sqrt{\pi})$ and $L^*,\;N^*$ are determined from \eqref{NL2}. 

We conclude that the solution of the problem \eqref{e6},\eqref{e7},\eqref{e8} and \eqref{e9} is
\begin{equation}\label{e10}
    u_2(\eta)=\dfrac{1+\alpha_0^{\nu}p^*\Phi(\eta, u_2(\eta))}{1+\alpha_0^{\nu}p^*\Phi(\xi, u_2(\xi))}
\end{equation}
with condition
\begin{equation}\label{e11}
    \dfrac{a\alpha_0^{\nu}E(\xi,u_2(\xi))\text{Ste}}{2\big[1+\alpha_0^{\nu}p^*\Phi(\xi,u_2(\xi))\big]}=\xi^{\nu+1}
\end{equation}
where $\Phi$ and $E$ are defined by \eqref{28} and \eqref{29}.

With help of \eqref{e1} and \eqref{e10} we summarize that solution of the problem \eqref{10},\eqref{11a},\eqref{12},\eqref{13}\eqref{14} can be represented in the form of
\begin{equation}\label{e12}
    \theta_2(\eta)=\theta^*+(\theta_m-\theta^*)u_2(\eta)
\end{equation}
where $\eta=z/(2\sqrt{t})$ and $u_2(\eta)$ satisfies the integral equation \eqref{e10} and condition \eqref{e11}.

\section{Existence and uniqueness of the similarity solution}
\subsection{Problem with heat flux condition}
To prove existence of the solution form \eqref{27} we assume that $\xi>0$ is a given constant. We consider the continuous real valued functions space $C^0[\alpha_0, \xi]$ which endowed with supremum norm 
$$||u||=\max_{\eta\in[\alpha_0,\xi]}|u(\eta)|$$
and using a fixed point Banach theorem $(C^0[\alpha_0,\xi],||\cdot||)$, We define operator $W: C^0[\alpha_0,\xi]\to C^0[\alpha_0,\xi]$ which is
\begin{equation}\label{31}
    W(u_2)(\eta):=u_2(\eta),\;\;\forall \eta\in[\alpha_0,\xi],
\end{equation}
where $u_2$ is defined by \eqref{27}. Then by using the fixed point Banach theorem we have to prove that operator \eqref{31} is contraction operator of mapping and it implies that there must exists unique solution $u\in C^0[\alpha_0,\xi]$ to integral solution \eqref{27}.

At first, we suppose that $L^*$ and $N^*$ are bounded and satisfy Lipschitz inequalities such that
\begin{enumerate}
\item[a)] There exists $L_m=\dfrac{\lambda_m}{\lambda_0}>0$ and $L_M=\dfrac{\lambda_M}{\lambda_0}>0$ such that
\begin{equation}\label{32}
L_m\leq L^*(u)\leq L_M,\;\;\;\forall u\in C^0(\mathbb{R}_0^+)\cup L^{\infty}(\mathbb{R}_0^+).
\end{equation}
and $\bar{L}=\dfrac{\bar{\lambda}(\theta_m+1)}{\lambda_0}>0$ such that
\begin{equation}\label{33}
||L^*(u_1)-L^*(u_2)||\leq \bar{L}||u_1-u_2||,\;\;\;\forall u_1,u_2\in C^0(\mathbb{R}_0^+)\cup L^{\infty}(\mathbb{R}_0^+).
\end{equation}
\item[b)] There exists $N_m=\dfrac{\sigma_m}{c_0,\gamma_0}>0$ and $N_M=\dfrac{\sigma_M}{c_0\gamma_0}>0$ such that
\begin{equation}\label{34}
N_m\leq N^*(u)\leq N_M,\;\;\;\forall u\in C^0(\mathbb{R}_0^+)\cup L^{\infty}(\mathbb{R}_0^+).
\end{equation}
and $\bar{N}=\dfrac{\bar{\sigma}(\theta_m+1)}{c_0\gamma_0}>0$ such that
\begin{equation}\label{35}
||N^*(u_1)-N^*(u_2)||\leq \bar{N}||u_1-u_2||,\;\;\;\forall u_1,u_2\in C^0(\mathbb{R}_0^+)\cup L^{\infty}(\mathbb{R}_0^+).
\end{equation}
\end{enumerate}
Now we have to obtain some preliminary results to prove the existence and uniqueness of the solution to the equation \eqref{27}.
\begin{lemma}\label{lem1}
For all $\eta\in[\alpha_0,\;\xi]$ the following inequality holds
\begin{equation}\label{36}
    \exp\bigg(-\dfrac{N_M}{aL_m}(\eta^2-\alpha_0^2)\bigg)\leq E(\eta, u)\leq \exp\bigg(-\dfrac{N_m}{aL_M}(\eta^2-\alpha_0^2)\bigg).
\end{equation}
\end{lemma}
\begin{proof}
    $E(\eta,u)\leq\exp\Bigg(-\dfrac{2N_m}{aL_M}\int\limits_{\alpha_0}^{\eta}sds\Bigg)=\exp\bigg(-\dfrac{N_m}{aL_M}(\eta^2-\alpha_0^2)\bigg)$.
\end{proof}
\begin{lemma}\label{lem2}
    For all $\eta\in[\alpha_0,\xi]$ the following inequality holds
    $$\dfrac{1}{2L_{M}}\exp\bigg(\tfrac{N_{M}}{aL_{m}}\alpha_0^2\bigg)\sqrt{\dfrac{N_{M}^{\nu-1}}{aL_{m}^{\nu-1}}}\bigg[\gamma\bigg(\dfrac{1-\nu}{2},\eta^2\dfrac{N_{M}}{aL_{m}}\bigg)-\gamma\bigg(\dfrac{1-\nu}{2},  \frac{N_{M}}{aL_{m}}\alpha_0^2\bigg)\bigg]\leq\Phi(\eta, u)$$
    $$\leq\dfrac{1}{2L_{m}}\exp\bigg(\frac{N_{m}}{aL_M}\alpha_0^2\bigg)\sqrt{\dfrac{N_{m}^{\nu-1}}{aL_{M}^{\nu-1}}}\bigg[\gamma\bigg(\dfrac{1-\nu}{2},\eta^2\dfrac{N_{m}}{aL_{M}}\bigg)-\gamma\bigg(\dfrac{1-\nu}{2},\frac{N_{m}}{aL_{M}}\alpha_0^2\bigg)\bigg],$$
\end{lemma}
\begin{proof}
We have $\Phi(\eta,u)\leq \dfrac{1}{L_{m}}\exp\bigg(\frac{N_{m}}{aL_{M}}\alpha_0^2\bigg)\int\limits_{\alpha_0}^{\eta}\dfrac{\exp(-N_{m}s^2/(aL_{M}))}{s^{\nu}}ds$ after using substitution $t=s\sqrt{\frac{N_{m}}{aL_{M}}}$ we obtain
$$\Phi(\eta,u)\leq\dfrac{1}{L_{m}}\exp\bigg(\frac{N_{m}}{aL_{M}}\alpha_0^2\bigg)\sqrt{\dfrac{N_{m}^{\nu-1}}{aL_{M}^{\nu-1}}}\int\limits_{\alpha_0\sqrt{N_{m}/(aL_{M})}}^{\eta\sqrt{N_{m}/(aL_{M})}}\dfrac{e^{-t^2}}{t^\nu}dt.$$
Then using substitution $z=t^{1-\nu}$ we get 
$$\Phi(\eta,u)\leq \dfrac{1}{L_{m}(1-\nu)}\exp\bigg(\frac{N_{m}}{aL_{M}}\alpha_0^2\bigg)\sqrt{\dfrac{N_{m}^{\nu-1}}{aL_{M}^{\nu-1}}}\int\limits_{(\alpha_0\sqrt{N_m/(aL_{M})})^{1-\nu}}^{(\eta\sqrt{N_{m}/(aL_{M})})^{1-\nu}}e^{-z^{\frac{2}{1-\nu}}}dz$$
and taking $y=z^{\frac{2}{1-\nu}}$ then inequality becomes
$$\Phi(\eta,u)\leq\dfrac{1}{L_{m}(1-\nu)}\exp\bigg(\frac{N_{m}}{aL_{M}}\alpha_0^2\bigg)\sqrt{\dfrac{N_{m}^{\nu-1}}{aL_{M}^{\nu-1}}}\dfrac{1-\nu}{2}\int\limits_{\alpha_0^2N_m/(aL_{M})}^{\eta^2N_{m}/(aL_{M}}y^{\frac{1-\nu}{2}-1}e^{-y}dy.$$
Then by using definition of special function type incomplete gamma function $\gamma(s,x)=\int\limits_0^x t^{s-1}e^{-t}dt$ we have proved that
$$\Phi(\eta, u)\leq\dfrac{1}{2L_{m}}\exp\bigg(\frac{N_{m}}{aL_{M}}\alpha_0^2\bigg)\sqrt{\dfrac{N_{m}^{\nu-1}}{aL_{M}^{\nu-1}}}\bigg[\gamma\bigg(\dfrac{1-\nu}{2},\eta^2\dfrac{N_{m}}{aL_{M}}\bigg)-\gamma\bigg(\dfrac{1-\nu}{2}, \frac{N_{m}}{aL_{M}}\alpha_0^2\bigg)\bigg].$$
\end{proof}
\begin{lemma}\label{lem3}
Let given $\alpha_0, \xi \in\mathbb{R^{+}}$ and assumptions \eqref{32},\eqref{33},\eqref{34},\eqref{35} hold for specific heat and dimensionless thermal conductivity then for all $u\in C^0[\alpha_0, \xi]$ we have
$$|E(\eta,u)-E(\eta,u^*)|\leq\dfrac{1}{aL_{m}}\bigg(\tilde{N}+\dfrac{N_{M}\tilde{L}}{L_{m}}\bigg)(\eta^2-\alpha_0^2)||u^*-u||.$$
\end{lemma}
\begin{proof}
By using inequality $\exp(-x)-\exp(-y)|\leq |x-y|,\;\;\forall x,y\geq 0$ we get
$$|E(\eta,u)-E(\eta,u^*)|\leq \Bigg|\exp\Bigg(-\dfrac{2}{a}\int\limits_{\alpha_0}^{\eta}s\dfrac{N(u_1(s))}{L(u(s))}ds\Bigg)-\exp\Bigg(-\dfrac{2}{a}\int\limits_{\alpha_0}^{\eta}s\dfrac{N(u^*(s))}{L(u^*(s))}ds\Bigg)\Bigg|$$
$$\leq \dfrac{2}{a}\Bigg|\int\limits_{\alpha_0}^{\eta}s\dfrac{N(u)}{L(u)}ds-\int\limits_{\alpha_0}^{\eta}s\dfrac{N(u^*)}{L(u^*)}ds\Bigg|\leq 2\int\limits_{\alpha_0}^{\eta}\Bigg|\dfrac{N(u)}{L(u)}-\dfrac{N(u^*)}{L(u^*)}\Bigg|sds$$
$$\leq\dfrac{2}{a}\int\limits_{\alpha_0}^{\eta}\Bigg|\dfrac{N(u)}{L(u)}-\dfrac{N(u^*)}{L(u)}+\dfrac{N(u^*)}{L(u)}-\dfrac{N(u^*)}{L(u^*)}\Bigg|sds$$
$$\leq\dfrac{2}{a}\int\limits_{\alpha_0}^{\eta}\Bigg(\dfrac{|N(u)-N(u^*)|}{|L(u)|}+\dfrac{|L(u^*)-L(u)|\cdot|N(u^*)|}{|L(u)||L(u^*)|}\Bigg)sds$$
$$\leq \dfrac{2}{aL_{m}}\bigg(\tilde{N}+\dfrac{N_{M}\tilde{L}}{L_{m}}\bigg)||u^*-u||\int\limits_{\alpha_0}^{\eta}sds=\dfrac{1}{aL_{m}}\bigg(\tilde{N}+\dfrac{N_{M}\tilde{L}}{L_{m}}\bigg)(\eta^2-\alpha_0^2)||u^*-u||.$$
\end{proof}

\begin{lemma}\label{lem4}
If $\alpha_0, \xi\in\mathbb{R^{+}}$ are given and \eqref{32}-\eqref{35} hold then for all $u^*\in C^0[\alpha_0, \xi]$ we have
$$|\Phi(\eta,u)-\Phi(\eta,u^*)|\leq\Tilde{\Phi}(\alpha_0,\xi)||u^*-u||,$$
where
\begin{equation}\label{37}
\Tilde{\Phi}(\alpha_0,\eta)=\dfrac{\alpha_0^{\nu}}{L_m^2}\bigg(\dfrac{1}{a}\bigg(\tilde{N}+\dfrac{N_{M}\tilde{L}}{L_{m}}\bigg)\bigg[\dfrac{\eta^{3-\nu}}{3-\nu}-\alpha_0^2\dfrac{\eta^{1-\nu}}{1-\nu}+\dfrac{2\alpha_0^{3-\nu}}{(3-\nu)(1-\nu)}\bigg]+\tilde{L}\dfrac{\eta^{1-\nu}-\alpha_0^{1-\nu}}{1-\nu}\bigg).
\end{equation}
\end{lemma}
\begin{proof}
By using lemmas \ref{lem2} and \ref{lem3} we obtain 
$$|\Phi(\eta,u)-\Phi(\eta,u^*)|\leq T_1(\eta)+T_2(\eta)$$
where 
$$T_1(\eta)\equiv\alpha_0^{\nu}\int\limits_{\alpha_0}^{\eta}\dfrac{|E(\eta,u)-E[\eta, u^*)|}{s^{\nu}L(u(s))}ds\leq \dfrac{\alpha_0^{\nu}}{aL_{m}^2}\bigg(\tilde{N}+\dfrac{N_{M}\tilde{L}}{L_{m}}\bigg)||u^*-u||\int\limits_{\alpha_0}^{\eta}(s^2-\alpha_0^2)s^{-\nu}ds$$
$$\leq \dfrac{\alpha_0^{\nu}}{aL_{m}^2}\bigg(\tilde{N}+\dfrac{N_{M}\tilde{L}}{L_{m}}\bigg)\bigg[\dfrac{\eta^{3-\nu}}{3-\nu}-\alpha_0^2\dfrac{\eta^{1-\nu}}{1-\nu}+\dfrac{2\alpha_0^{3-\nu}}{(3-\nu)(1-\nu)}\bigg]||u^*-u||$$
and
$$T_2(\eta)\equiv \alpha_0^{\nu}\int\limits_{\alpha_0}^{\eta}\bigg|\dfrac{1}{L(u)}-\dfrac{1}{L(u^*)}\bigg|\dfrac{1}{s^{\nu}}\exp\Bigg(-\dfrac{2}{a}\int\limits_{\alpha_0}^{\eta}t\dfrac{N(u^*)}{L(u^*)}dt\Bigg)ds$$
$$\leq \alpha_0^{\nu}\int\limits_{\alpha_0}^{\eta}\dfrac{|L(u^*)-L(u)|}{|L(u)||L(u^*)|}\dfrac{ds}{s^{\nu}}\leq \dfrac{\tilde{L}\alpha_0^{\nu}}{L_{m}^2}||u^*-u||\int\limits_{\alpha_0}^{\eta}\dfrac{ds}{s^{\nu}}\leq \dfrac{\tilde{L}(\eta^{1-\nu}-\alpha_0^{1-\nu})\alpha_0^{\nu}}{L_{m}^2(1-\nu)}||u^*-u||. $$
Finally we get
$$T_1(\eta)+T_2(\eta)\leq \dfrac{\alpha_0^{\nu}}{L_{m}^2}||u^*-u||\bigg(\dfrac{1}{a}\bigg(\tilde{N}+\dfrac{N_M\tilde{L}}{L_{m}}\bigg)\bigg[\dfrac{\eta^{3-\nu}}{3-\nu}-\alpha_0^2\dfrac{\eta^{1-\nu}}{1-\nu}+\dfrac{2\alpha_0^{3-\nu}}{(3-\nu)(1-\nu)}\bigg]+\tilde{L}\dfrac{\eta^{1-\nu}-\alpha_0^{1-\nu}}{1-\nu}\bigg).$$
\end{proof}

\begin{theorem}\label{th1}
Suppose that $L^*$ and $N^*$ satisfy the conditions \eqref{32}-\eqref{35}. If $\alpha_0<\xi<\xi^*$ where $\xi^*>0$ is defined as unique solution to $\epsilon(\alpha_0, z)=1$ with
\begin{equation}\label{38}
\epsilon(\alpha_0,z):=2p^*\tilde{\Phi}(\alpha_0,z)
\end{equation}
where $\tilde{\Phi}(\alpha_0,\eta)$ is given by \eqref{37}, then there exists a unique solution $u_2\in C^0[\alpha_0,\mu]$ for the integral equation \eqref{27}.
\end{theorem}
\begin{proof}
We have to show that operator $W$ is defined by \eqref{31} is a contraction operator. Suppose we have $u_2, u_2^*\in C^0[\alpha_0, \xi]$ and by using lemmas \ref{lem1}-\ref{lem4} we have
\begin{equation}
\begin{split}
    &|W(u_2(\eta))-W(u_2^*(\eta))|\leq q^*|\Phi(\xi, u_2(\xi))-\Phi(\eta, u_2(\eta))-\Phi(\xi, u_2^*(\xi))+\Phi(\eta,u_2^*(\eta))|\\
    &\leq q^*(|\Phi(\xi,u_2(\xi))-\Phi(\xi,u_2^*(\xi))|+|\Phi(\eta,u_2(\eta))-\Phi(\eta,u_2^*(\eta))|)\\
    &\leq 2q^*\Tilde{\Phi}(\alpha_0,\xi)||u_2-u_2^*||.
\end{split}
\end{equation}
It follows that 
$$|W(u_2)(\eta)-W(u_2^*)(\eta)|\leq \epsilon(\alpha_0,\xi) ||u_2-u_2^*||$$
where $\epsilon(\alpha_0, \eta)$ is defined by \eqref{38} and we can notice that
$$\epsilon(\alpha_0,\alpha_0)<1,\;\;\forall\xi: \alpha_0<\xi<\xi^*,\;\;\;\epsilon(\alpha_0,\xi)>1,\;\;\forall \xi: \xi>\xi^*.$$
Then we can make conclusion that $\epsilon$ is increasing function and thus there exists a unique $\xi^*>0$ such that $\epsilon(\alpha_0,\xi^*)=1$ so the operator $W$ becomes a contraction operator of mapping. By the fixed point Banach theorem there must exist a unique solution $u_2\in C^0[\alpha_0,\xi]$ to
integral equation \eqref{27}.
\end{proof}
Now we analyze the existence and uniqueness of the solution for the equation \eqref{30}. We have to show that 
\begin{equation}\label{39}
\phi(\xi)=\xi^{\nu+1}
\end{equation}
where 
$$\phi(\xi)= \dfrac{q^*\alpha_0^{\nu}E(\xi,u_2(\xi))}{M\lambda(\theta_m)},$$
$$q^*=\dfrac{P_0e^{-\alpha_0^2}}{\alpha_0\theta_m\sqrt{\pi}},$$ $$M=\dfrac{2l_m\gamma_m}{\lambda_0\theta_m\lambda(\theta_m)}$$
has a unique solution $\xi\in[\alpha,\xi^*]$.

\begin{lemma}\label{lem5}
Suppose assumptions \eqref{32}-\eqref{35} hold, then for all $\xi\in[\alpha_0,\xi^*]$ we have that
\begin{equation}\label{40}
    \phi_1(\xi)\leq\phi(\xi)\leq\phi_2(\xi)
\end{equation}
where $\phi_1(\xi)$ and $\phi_2(\xi)$ are functions defined by
\begin{equation}\label{41}
    \begin{split}
        &\phi_1(\xi)=\dfrac{q^*\alpha_0^{\nu}}{M\lambda(\theta_m)}\exp\bigg(-\dfrac{N_M}{aL_M}(\xi-\alpha_0)\bigg),\;\;\;\;\;\;\;\;\;\;\;\;\;\;\;\;\;\xi>\alpha_0,\\
        &\phi_2(\xi)=\dfrac{q^*\alpha_0^{\nu}}{M\lambda(\theta_m)}\exp\bigg(\dfrac{N_m}{aL_M}(\xi^*-\alpha_0)-\dfrac{N_M}{aL_m}(\xi-\alpha_0)\bigg),\;\;\xi>\alpha_0
    \end{split}
\end{equation}
which satisfy the following properties
\begin{equation}\label{42}
    \begin{split}
        &\phi_1(\alpha_0)=\dfrac{q^*\alpha_0^{\nu}}{M\lambda(\theta_m)}>0,\;\;\;\phi_1(+\infty)=0,\;\;\;\phi_1'(\xi)<0,\;\;\forall\xi>\alpha_0\\
        &\phi_2(\alpha_0)=\dfrac{q^*\alpha_0^{\nu}}{M\lambda(\theta_m)}>0,\;\;\;\phi_2(+\infty)=0,\;\;\;\phi_2'(\xi)<0,\;\;\forall\xi>\alpha_0.
    \end{split}
\end{equation}
\end{lemma}
\begin{proof}
We can easily prove this lemma directly using bound \eqref{36} and definitions \eqref{41},\eqref{42} of the functions $\phi_1$ and $\phi_2$.
\end{proof}

\begin{lemma}\label{lem6}
If
\begin{equation}\label{eq1}
    \phi_2(\xi^*)<\xi^*
\end{equation}
then, there exists a unique solution $\alpha_0<\xi_1<\xi^*$ to the equation
\begin{equation}\label{43}
\phi_1(\xi)=\xi^{\nu+1},\;\;\xi>\alpha_0
\end{equation}
and there exists a unique solution $\xi_1<\xi_2<\xi^*$ to the equation
\begin{equation}\label{44}
\phi_2(\xi)=\xi^{\nu+1},\;\;\xi>\alpha_0.
\end{equation}
\end{lemma}
\begin{proof}
We can prove by using properties of  $\phi_1$ and $\phi_2$ shown in Lemma \ref{lem5}.
\end{proof}

\begin{remark}
By using definition of $\phi_2$ and $M$ we obtain that assumption \eqref{eq1} is equivalent to the following inequality for latent of melting heat
\begin{equation}\label{eq2}
    l_m>\dfrac{q^*\alpha_0^{\nu}\lambda_0\theta_m}{2\gamma_m\xi^*}\exp\bigg(\dfrac{N_m}{aL_M}(\xi^*-\alpha_0)-\dfrac{N_M}{aL_m}(\xi^*-\alpha_0)\bigg).
\end{equation}
\end{remark}

\begin{theorem}\label{th2}
Suppose \eqref{32}-\eqref{35} and \eqref{eq2} hold. Consider $\xi_1$ and $\xi_2$ determined from \eqref{43} and \eqref{44}. If $\epsilon(\alpha_0,\xi_2)<1$, where $\epsilon$ is defined by \eqref{38}, then there exists at least one solution $\bar{\xi}\in(\xi_1,\xi_2)$ to the equation \eqref{30}. 
\end{theorem}

\begin{proof}
By hypothesis of Lemma \ref{lem5} if $\epsilon(\alpha_0,\xi_2)<1$ then we have that the inequality \eqref{40} holds for each $\xi_1\leq\xi\leq\xi_2\leq\xi^*$ and $\epsilon(\alpha_0,\xi)<1$. As function $\phi$ is continuous decreasing function we obtain that there exists at least one solution $\bar{\xi}\in[\xi_1,\xi_2]$ to the equation \eqref{30}.
\end{proof}

Now we can make conclusion by following main theorem.
\begin{theorem}\label{th3}
Assume  that \eqref{32}-\eqref{25} hold and $\epsilon(\alpha_0,\xi_2)<1$ where $\epsilon$ defined by \eqref{38} and $\xi_2$ defined from \eqref{44} then there exist at least one solution to the problem \eqref{10}-\eqref{14} where unknown free boundary is given by
\begin{equation}\label{45}
\beta(t)=2\bar{\xi}\sqrt{t},\;\;t>0
\end{equation}
where $\bar{\xi}$ defined from Theorem \ref{th2} and temperature is given by
\begin{equation}\label{46}
\theta(z,t)=\theta_m(u_{\bar{\xi}}(\eta)+1),\;\;\alpha_0\leq\eta\leq\bar{\xi}
\end{equation}
where $\eta=\dfrac{z}{2\sqrt{t}}$ being similarity substitution and $u_{\bar{\xi}}$ is the unique solution of the integral equation \eqref{27} which was established in Theorem \ref{th1}.  
\end{theorem}

\subsection{Problem with convective boundary condition} In this section, analogously as in previous, we will prove existence and uniqueness of the solution form \eqref{e10} assuming that there is given constant $\xi>0$ and considering fixed point Banach space $(C^0[\alpha_0,\xi],||\cdot||)$, defining the operator $V:C^0[\alpha_0,\xi]\to C^0[\alpha_0,\xi]$ such
\begin{equation}\label{e13}
    V(u_2)(\eta)=u_2(\eta),\;\;\;\alpha_0\leq \eta\leq \xi,
\end{equation}
where $u_2$ is defined by \eqref{e10}. 

Let assume that $L^*, N^*$ satisfy all assumptions \eqref{32}-\eqref{35} then we can get the following results.

\begin{theorem}\label{th4}
Suppose that \eqref{32}-\eqref{35} hold. If $\alpha_0\leq \xi\leq \xi_{c}^*$ where $\xi_c^*$ is defined as the unique solution of $\widehat{\epsilon}(\alpha_0, z)=1$ such as
\begin{equation}\label{e14}
    \widehat{\varepsilon}(\alpha_0, z):=\dfrac{\tilde{\Phi}(\alpha_0, z)}{1+\frac{\alpha_0^{\nu}p^*}{2L_m}\exp\bigg(\alpha_0^2\frac{N_m}{aL_M}\bigg)\sqrt{\frac{N_m^{\nu-1}}{aL_M^{\nu-1}}}h(\alpha_0,z)},
\end{equation}
where $\tilde{\Phi}(\alpha_0,z)$ defined from \eqref{37} and 
$$h(\alpha_0,\eta)=\gamma\bigg(\dfrac{1-\nu}{2},\eta^2\dfrac{N_{m}}{aL_{M}}\bigg)-\gamma\bigg(\dfrac{1-\nu}{2}, \frac{N_{m}}{aL_{M}}\alpha_0^2\bigg),$$
then there exists a unique solution $u_2\in C^0[\alpha_0,\xi]$ for integral equation \eqref{e10}.
\end{theorem}

\begin{proof}
By analogously approach in Theorem \ref{th1}, we need to show that operator $V$ defined by \eqref{e13} is a contraction operator and we suppose that there exists $u_2,\;u_2^*\in C^0[\alpha_0,\xi]$ then by using lemmas \ref{lem1}-\ref{lem4} we get
$$||V(u_2)(\eta)-V(u_2^*)(\eta)||\leq \max_{\eta\in[\alpha_0,\xi]}\Bigg|\dfrac{1+\alpha_0^{\nu}p^*\Phi(\eta,u_2)}{1+\alpha_0^{\nu}p^*\Phi(\xi,u_2)}-\dfrac{1+\alpha_0^{\nu}p^*\Phi(\eta,u_2^*)}{1+\alpha_0^{\nu}p^*\Phi(\xi,u_2^*)}\Bigg|$$
$$\leq \max_{\eta\in[\alpha_0,\xi]} \dfrac{\bigg|(1+\alpha_0^{\nu}p^*\Phi(\eta,u_2))(1+\alpha_0^{\nu}p^*\Phi(\xi,u_2^*))-(1+\alpha_0^{\nu}p^*\Phi(\eta,u_2^*))(1+\alpha_0^{\nu}p^*\Phi(\xi,u_2))\bigg|}{\bigg|1+\alpha_0^{\nu}p^*\Phi(\xi,u_2)\bigg|\bigg|1+\alpha_0^{\nu}p^*\Phi(\xi,u_2^*)\bigg|}$$
$$\leq \widehat{\varepsilon}(\alpha_0, \xi)||u_2^*-u_2||,$$
where $\widehat{\varepsilon}(\alpha_0,\xi)$ defined by \eqref{e14} and it is easy to check that
$$\widehat{\varepsilon}(\alpha_0,\alpha_0)<1, \;\;\forall\xi\in [\alpha_0, \xi_c^*],\;\;\;\;\widehat{\varepsilon}(\alpha_0,\xi)>1,\;\;\forall\xi\in[\xi_c^*,\infty).$$
We can see that $\widehat{\varepsilon}$ is an increasing function then it enables us to make conclusion that there exists a unique positive constant $\xi_c^*$ such that $\widehat{\varepsilon}(\alpha_0, \xi_c^*)=1$ and we obtain that operator $V$ is a contraction mapping operator. At the end, we can make conclusion that there must be a unique solution $u_2\in C^0[\alpha_0,\xi]$ to the equation \eqref{e10}.
\end{proof}
We obtained that for each given $\alpha_0<\xi<\xi_c^*$, a unique solution for \eqref{e10} is $u_2(\eta)=u_{2(\xi)}(\eta)$ and its derivative will be
\begin{equation}\label{e15}
    u_{2(\xi)}'(\eta)=\dfrac{\alpha_0^{\nu}p^*E(\eta, u_{2(\xi)}(\eta))}{\big[1+\alpha_0^{\nu}p^*\Phi(\eta, u_{2(\xi)}(\eta))\big]\eta^{\nu}L^*(u_{2(\xi)(\eta)}).}
\end{equation}
It remains to analyze the condition \eqref{e11} which can be rewritten as
\begin{equation}\label{e16}
    \phi^c(\xi)=\varphi^c(u_{2(\xi)},\xi):=\xi^{\nu+1}
\end{equation}
where
$$\phi^c(\xi)=\dfrac{a\alpha_0^{\nu}E(\xi,u_2(\xi))\text{Ste}}{2\big[1+\alpha_0^{\nu}p^*\Phi(\xi,u_2(\xi))\big]}.$$
Then we can obtain the next results.
\begin{lemma}\label{lem7}
Assume that \eqref{32}-\eqref{35} hold. Then for all $\xi\in(\alpha_0,\xi_c^*)$ we have
\begin{equation}\label{e17}
    0\leq \phi^c(\xi)\leq \phi_2(\xi)
\end{equation}
where $\phi_2$ is defined by \eqref{41}.  
\end{lemma}
\begin{proof}
The proof follows straightforwardly by taking into account the bounds given in Lemma \ref{lem1} and definition of $\phi_2$ in \eqref{41}.

Similarly, we notice that the properties of $\phi_2(\xi)$ studied in Lemma \ref{lem6} and if \eqref{44} holds then there exists unique solution $\alpha_0<\xi_2\leq\xi_c^*$ for the equation \eqref{e16}.  
\end{proof}

\begin{theorem}\label{th5}
Assume that \eqref{32}-\eqref{35}, \eqref{44} hold, then by Lemma \ref{lem6} we can conclude that there exists unique solution $\widetilde{\xi}_c^*\in(\alpha_0, \xi_2)$ to the equation \eqref{e16}.
\end{theorem}
\begin{proof}
It can be proved analogously as Theorem \ref{th2}.
\end{proof}
\begin{theorem}\label{th6}
Assume that \eqref{32}-\eqref{35}, \eqref{e17} hold, then there exists at least one solution of the problem \eqref{10}-\eqref{14} with replaced condition with \eqref{11a} where free boundary is defined by
\begin{equation}\label{e18}
    \beta(t)=2\widetilde{\xi}_c^*\sqrt{t},\;\;t>0,
\end{equation}
where $\widetilde{\xi}_c^*$ is defined in Theorem \ref{th5} and temperature in liquid region is given by
\begin{equation}\label{e19}
    \theta_2(z,t)=(\theta_m-\theta^*)u_{2(\widetilde{\xi}_c^*)}(\eta)+\theta^*,\;\;\;\alpha_0\leq \eta\leq \widetilde{\xi}_c^*,
\end{equation}
where $\eta=z/(2\sqrt{t})$ is similarity variable and $u_{2(\widetilde{\xi}_c^*)}$ is unique solution to the integral equation \eqref{e10} established from Theorem \ref{th4}.
\end{theorem}

\section{Particular cases for thermal conductivity}
\subsection{Constant thermal coefficients}
In this section we are going to analyze the solution \eqref{27} and \eqref{e10} when thermal coefficients are constant such that
\begin{equation}\label{e20}
    c(\theta_2)=c_0,\;\;\;\rho(\theta_2)=\rho_0,\;\;\;\lambda(\theta_2)=\lambda_0,
\end{equation}
then replacing \eqref{NL2} with $N^*=L^*=1$ we get the results for $E$ and $\Phi$ functions as the following
\begin{equation}\label{e21}
    E(\eta, u_2(\eta))=\exp\bigg(-\dfrac{1}{a}(\eta^2-\alpha_0^2)\bigg),
\end{equation}
\begin{equation}\label{e22}
    \Phi(\eta,u_2(\eta))=\dfrac{1}{2}\exp\bigg(\dfrac{\alpha_0^2}{a}\bigg)a^{\frac{1-\nu}{2}}\bigg[\gamma\bigg(\dfrac{1-\nu}{2},\dfrac{\eta^2}{a}\bigg)-\gamma\bigg(\dfrac{1-\nu}{2}, \dfrac{\alpha_0^2}{a}\bigg)\bigg].
\end{equation}
By making substitutions the \eqref{e21}, \eqref{e22} into integral equations \eqref{27} and \eqref{e10} then we have solution for the problem with heat flux condition as
\begin{equation}\label{e23}
    u_2(\eta)=\dfrac{q^*}{2}\exp\bigg(\dfrac{\alpha_0^2}{a}\bigg)a^{\frac{1-\nu}{2}}\bigg[\gamma\bigg(\dfrac{1-\nu}{2},\dfrac{\xi^2}{a}\bigg)-\gamma\bigg(\dfrac{1-\nu}{2},\dfrac{\eta^2}{a}\bigg)\bigg]
\end{equation}
with condition
\begin{equation}\label{e24}
\phi(\xi)=\xi^{\nu+1}
\end{equation}
where 
\begin{equation}\label{e25}
    \phi(\xi)=\dfrac{q^*\alpha_0^{\nu}\lambda_0\exp\big(-\frac{1}{a}(\xi^2-\alpha_0^2)\big)}{2l_m\gamma_m}
\end{equation}
and it is easy to check that function \eqref{e25} is decreasing function such that
$$\phi(\alpha_0)>0,\;\;\;\phi(+\infty)=0,\;\;\;\phi'(\xi)<0,$$
then we can state that equation \eqref{e24} has an unique solution.

With help of \eqref{e21} and \eqref{e22} the solution of the problem \eqref{10}-\eqref{14} replaced with condition \eqref{11a} can be represented
\begin{equation}\label{e26}
    u_2(\eta)=\dfrac{1+\frac{\alpha_0^{\nu}p^*}{2}\exp\bigg(\dfrac{\alpha_0^2}{a}\bigg)a^{\frac{1-\nu}{2}}\bigg[\gamma\bigg(\dfrac{1-\nu}{2},\dfrac{\eta^2}{a}\bigg)-\gamma\bigg(\dfrac{1-\nu}{2}, \dfrac{\alpha_0^2}{a}\bigg)\bigg]}{1+\frac{\alpha_0^{\nu}p^*}{a}\exp\bigg(\dfrac{\alpha_0^2}{a}\bigg)a^{\frac{1-\nu}{2}}\bigg[\gamma\bigg(\dfrac{1-\nu}{2},\dfrac{\xi^2}{a}\bigg)-\gamma\bigg(\dfrac{1-\nu}{2}, \dfrac{\alpha_0^2}{a}\bigg)\bigg]}
\end{equation}
with condition
\begin{equation}\label{e27}
    \phi_c(\xi)=\xi^{\nu+1}
\end{equation}
where
\begin{equation}\label{e28}
    \phi_c(\xi)=\dfrac{a\alpha_0^{\nu}\exp\big(-\frac{1}{a}(\xi^2-\alpha_0^2)\big)\text{Ste}}{2\bigg[1+\frac{\alpha_0^{\nu}p^*}{a}\exp\big(\frac{\alpha_0^2}{a}\big)a^{\frac{1-\nu}{2}}\bigg(\gamma\big(\frac{1-\nu}{2},\frac{\xi^2}{a}\big)-\gamma\big(\frac{1-\nu}{2}, \frac{\alpha_0^2}{a}\big)\bigg)\bigg]}
\end{equation}
and here we can also see that function \eqref{e28} is non-increasing function because
$$\phi_c(\alpha_0)>0,\;\;\;\phi_c(+\infty)=0,\;\;\;\phi_c'(\xi)<0,$$
then we can be obtained that there exists a unique solution to equation \eqref{e27}.

\subsection{Linear thermal coefficients}
In this subsection we are going to analyse the case when thermal coefficients are given by
\begin{equation}\label{e29}
    c(\theta_2)=c_0\bigg(1+\alpha\dfrac{\theta-\theta^*}{\theta_m-\theta^*}\bigg),\;\;\;\rho(\theta_2)=\rho_0,\;\;\;\lambda(\theta_2)=\lambda_0\bigg(1+\beta\dfrac{\theta-\theta^*}{\theta_m-\theta^*}\bigg)
\end{equation}
where $\alpha$ and $\beta$ are given positive constants. This particular case can be considered in this paper for the problem \eqref{10},\eqref{11a},\eqref{12}-\eqref{14}. 

From \eqref{NL2} replacing \eqref{e29} we can obtain
$$L*(u_2)=1+\beta u_2,\;\;\;N^*(u_2)=1+\alpha u_2$$
and notice that $u_2\in C^0[\alpha_0,\xi]$ then taking $\alpha_0=1,\;\xi=2$ from assumptions \eqref{32}-\eqref{35} we get the 
$$1+\beta\leq L^*(u_2)\leq 1+2\beta,\;\;\;1+\alpha\leq N^*(u_2)\leq 1+2\alpha$$
with 
$$L_m=1+\beta,\;\;\;L_M=1+2\beta,\;\;\;N_m=1+\alpha,\;\;\;N_M=1+2\alpha.$$
Then definition of $E$ and $\Phi$ functions becomes
\begin{equation}\label{e30}
    E(\eta,u_2(\eta))=\exp\bigg(-\dfrac{1+\alpha}{a(1+\beta)}(\eta^2-\alpha_0^2)\bigg),
\end{equation}
\begin{equation}\label{e31}
    \begin{split}
        \Phi(\eta,u_2(\eta))&=\dfrac{1}{2(1+\beta)}\exp\bigg(\frac{1+\alpha}{a(1+2\beta)}\alpha_0^2\bigg)\sqrt{\dfrac{(1+\alpha)^{\nu-1}}{a(1+2\beta)^{\nu-1}}}\\
        &\cdot\bigg[\gamma\bigg(\dfrac{1-\nu}{2},\eta^2\dfrac{1+\alpha}{a(1+2\beta)}\bigg)-\gamma\bigg(\dfrac{1-\nu}{2}, \frac{1+\alpha}{a(1+2\beta)}\alpha_0^2\bigg)\bigg].    
    \end{split}
\end{equation}
By using \eqref{e30} and \eqref{e31} the integral equation \eqref{e10} can be rewritten as the following form
\begin{equation}\label{e32}
    u_2(\eta)=\dfrac{1+\frac{\alpha_0^{\nu}p^*}{2(1+\beta)}\exp\bigg(\frac{\alpha_0^2(1+\alpha)}{a(1+2\beta)}\bigg)\sqrt{\frac{(1+\alpha)^{\nu-1}}{a(1+2\beta)^{\nu-1}}}w(\alpha_0,\eta)}{1+\frac{\alpha_0^{\nu}p^*}{2(1+\beta)}\exp\bigg(\frac{\alpha_0^2(1+\alpha)}{a(1+2\beta)}\bigg)\sqrt{\frac{(1+\alpha)^{\nu-1}}{a(1+2\beta)^{\nu-1}}}w(\alpha_0,\xi)}
\end{equation}
where
$$w(\alpha_0,\eta)=\bigg[\gamma\bigg(\dfrac{1-\nu}{2},\eta^2\dfrac{1+\alpha}{a(1+2\beta)}\bigg)-\gamma\bigg(\dfrac{1-\nu}{2}, \frac{1+\alpha}{a(1+2\beta)}\alpha_0^2\bigg)\bigg]$$
with condition
\begin{equation}\label{e33}
    \widetilde{\phi}_c(\xi)=\xi^{\nu+1}
\end{equation}
where
\begin{equation}\label{e34}
    \widetilde{\phi}_c(\xi)=\dfrac{a\alpha_0^{\nu}\exp\big(-\frac{1+\alpha}{a(1+\beta)}(\xi^2-\alpha_0^2)\big)\text{Ste}}{2\big[1+\alpha_0^{\nu}p^*\Phi(\xi,u_2(\xi))\big]}
\end{equation}
with $\Phi(\xi,u_2(\xi))$ which can be defined by \eqref{e31}.
We can easily notice that function $\widetilde{\phi}_c$ is a decreasing function for all $\eta\in(\alpha_0,\xi)$ and it enables us to get statement that equation \eqref{e33} has a unique solution.

\section*{Conclusion}
We have studied one-phase Stefan problem for generalized heat equation with heat flux entering to domain $D_2$ from metallic vapour zone through free boundary $z=\alpha(t)$ which determined from \eqref{9}. The temperature field in liquid metal zone and free boundary on melting interface are determined. Existence and uniqueness of the similarity solution imposing heat flux and convective boundary condition at the known left free boundary which describes the location of the boiling interface is proved. This article will be very useful in electrical contact engineers to describe heat process arising in the body with cross-section variable regions, in particular, the metal bridge between two electrical contact materials is melted when explosion appears and to avoid from crashing contacts it is very important to analyze the heat transfer in bridge material with different characteristics. Explicit solutions for the problem \eqref{10}-\eqref{14} with constant and linear thermal coefficients are represented, existence and uniqueness of the solution is successfully discussed. 

\section*{Acknowledgment}
The author thanked to prof. S.N. Kharin for supporting and valuable comments. The present work has been sponsored by the grant project AP14869306 "Special methods for solving electrical contact Stefan type problems and their application to the study of electric arc processes" from the Ministry of Science and Education of the Republic of Kazakhstan.

\end{document}